\documentclass[12pt]{amsart}
\usepackage{amsfonts}
\usepackage{mathrsfs}
\usepackage{amscd,amsmath,amssymb,amsfonts}
\usepackage[all]{xy}
\theoremstyle{plain}
\newtheorem{thm}{Theorem}
\newtheorem{lem}[thm]{Lemma}

\newtheorem{prop}[thm]{Proposition}

\theoremstyle{definition}
\newtheorem{defn}[thm]{Definition}

\newtheorem{rmk}[thm]{Remark}

\newtheorem{claim}[thm]{Claim}

\numberwithin{thm}{section} \numberwithin{equation}{section}

\newcommand{\ga}[2]{\begin{gather}\label{#1}#2 \end{gather}}

\newcommand{\surj}{\twoheadrightarrow}

\newcommand{\Pic}{{\rm Pic}}

\newcommand{\Spec}{{\rm Spec \,}}

\newcommand{\sD}{{\mathcal D}}
\newcommand{\sE}{{\mathcal E}}

\newcommand{\sH}{{\mathcal H}}

\newcommand{\sL}{{\mathcal L}}

\newcommand{\sN}{{\mathcal N}}
\newcommand{\sO}{{\mathcal O}}

\newcommand{\C}{{\mathbb C}}

\newcommand{\E}{{\mathbb E}}
\newcommand{\F}{{\mathbb F}}
\newcommand{\G}{{\mathbb G}}

\newcommand{\N}{{\mathbb N}}

\newcommand{\Z}{{\mathbb Z}}

\begin{document}

\title{Stratified bundles and \'etale fundamental group}
\author{H{\'e}l{\`e}ne Esnault}
\author{Xiaotao Sun}
\address{Mathematik, Universit{\"a}t Duisburg-Essen, 45117 Essen, Germany}
\email{esnault@uni-due.de}
\address{Academy of Mathematics and Systems Science, Chinese Academy of Science,
Beijing, P. R. of China}
\email{xsun@math.ac.cn}
\date{December 20, 2011}
\thanks{The first author is supported by  the SFB/TR45
and the ERC Advanced Grant 226257. The second author is supported by
NBRPC 2011CB302400, NSFC60821002/F02 and NSFC 10731030}
\begin{abstract} On $X$ projective smooth over an algebraically closed field of
characteristic $p>0$, we show that all irreducible stratified
bundles have rank $1$ if and only if the commutator  $[\pi_1,
\pi_1]$ of the \'etale fundamental group $\pi_1$ is a pro-$p$-group,
and we show that the category of stratified bundles is semi-simple
with irreducible objects of rank $1$ if and only if $ \pi_1 $  is
abelian without $p$-power quotient. This answers positively a
conjecture by Gieseker \cite[p.~8]{Gi}.
\end{abstract}

\maketitle
\begin{quote}

\end{quote}
\section{Introduction}
Let $X$ be a smooth projective variety defined over an
algebraically closed
field $k$ of characteristic $p>0$. In \cite{Gi},
stratified bundles are defined and studied. It is shown that they are a
characteristic $p>0$ analog to complex local systems over smooth complex
algebraic varieties. In particular, Gieseker
shows \cite[Theorem~1.10]{Gi}
\begin{itemize}
\item[(i)] If every stratified bundle is trivial,
then $\pi_1$ is trivial.
\item[(ii)] If all the irreducible stratified bundles have rank $1$,
then
$[\pi_1,\pi_1]$ is a pro-$p$-group.
\item [(iii)] If every stratified bundle is a direct sum of stratified
line bundles, then $\pi_1$ is abelian without non-trivial   $p$-power quotient.
\end{itemize}
Here $\pi_1$ is the \'etale fundamental group based at some
geometric point. He conjectures that in the three statements, the
``\,if\,'' can be replaced by ``\,if and only if\,''. The aim of
this note is to give a positive answer to Gieseker's conjecture (see
Theorem~\ref{thm3.9}).

The converse to (i) is the main theorem of \cite{EM},  and is analog
to Malcev-Grothendieck theorem  (\cite{Mal}, \cite{Gr}) asserting
that if the \'etale fundamental group of a smooth complex projective
variety is trivial, then there is no non-trivial bundle with a flat
connection.

The complex analog to the converse to (iii) relies on the following.

If $X$ is an abelian variety over a field $k$, Mumford
\cite[Section~16]{Mu} showed that a non-trivial line bundle $L$
which is algebraically equivalent to $0$ fulfills $H^i(X, L)=0$ for
all $i\ge 0$.  If $k=\C$, the field of complex numbers, $L$ carries
a unitary flat connection with underlying local system $\ell$, and
$H^1(X, L)=0$ implies that $H^1(X, \ell)=0$. In fact, more
generally, if $X$ is any complex manifold with abelian topological
fundamental group, and $\ell$ is a non-trivial rank $1$ local
system, then $H^1(X, \ell)=0$ (see Remark~\ref{rmk3.11}).

Finally the complex analog to the converse to (ii) and (iii)
together says that if $X$ is a smooth complex variety, then all
irreducible complex local systems on $X$ have rank $1$ if and only
if $\pi_1$ is abelian. However, the semi-simplicity statement has a
different phrasing: all irreducible complex local systems on $X$
have rank $1$ and the category is semi-simple if and only if $\pi_1$
is a finite abelian group (see Claim~\ref{claim3.13} for a precise
statement).

The proof of the converse to (ii) is done in Section 2. It relies on
a consequence of the proof of the main theorem in \cite{EM}, which
is formulated in Theorem~\ref{thm2.3}. It is a replacement for the
finite generation of the topological fundamental group in complex
geometry (see the analog statement in Remark~\ref{rmk3.12}). The
proof of the converse to (iii) is done in Section 3. The key point
to show the converse to (iii) is Theorem~\ref{thm3.4}. Its proof
relies on the construction (Proposition~\ref{prop3.5}) of a moduli
space of non-trivial extensions, with a rational map $f$ induced by
Frobenius pullback, which allows us to
use a theorem of Hrushovski. \\[.2cm]
{\it Acknowledegements:} We thank the department of Mathematics of
Harvard University for its hospitality during the preparation of
this note. We thank Nguy\~{\^e}n Duy T\^an for a very careful
reading of a first version, which allowed us to improve the
redaction. We thank the referee of the first version
who pointed out a mistake, corrected in Section~3. The second named
author would like to thank his colleague Nanhua Xi for discussions
related to Claim~\ref{claim3.3}.

\section{Stratified bundles}

Let $k$ be an algebraically closed field of characteristic $p>0$,
and $X$ a smooth connected projective variety over $k$. A stratified
bundle on $X$ is by definition a coherent $\sO_X$-module $\sE$ with
a homomorphism $$\nabla: \sD_X\to \sE nd_k(\sE)$$ of
$\sO_X$-algebras, where $\sD_X$ is the sheaf of differential
operators acting on the structure sheaf of $X$. By a theorem of Katz
(cf. \cite[Theorem 1.3]{Gi}), it is equivalent to the following
definition. (The terminology is not unique: stratified bundles from
the following definition are called flat bundles  in \cite{Gi},
$F$-divided sheaves in \cite{dS}).
\begin{defn}\label{defn2.1}
A stratified bundle on $X$ is a sequence of bundles $$E=\{E_0, E_1,
E_2, \cdots, \sigma_0, \sigma_1,\ldots \}=\{E_i,\sigma_i\}_{i\in\Bbb
N}$$ where $\sigma_i:F_X^*E_{i+1}\to E_i$  is a $\sO_X$-linear
isomorphism, and  $F_X:X\to X$ is the absolute Frobenius.
\end{defn}
A morphism $\alpha=\{\alpha_i\}:\{E_i,\sigma_i\}\to \{F_i,\tau_i\}$
between two stratified bundles is a sequence of morphisms $\alpha_i:
E_i\to F_i$ of $\sO_X$-modules such that
$$\xymatrix{
   F_X^*E_{i+1} \ar[d]_{\sigma_i} \ar[r]^{ F_X^*\alpha_{i+1}}
                &  \ar[d]^{\tau_i}  F_X^*F_{i+1}\\
  E_i  \ar[r]^{\alpha_i}
                &             F_i }$$
is commutative. The category $\textbf{str}(X)$ of stratified bundles
is abelian, rigid, monoidal. To see it is $k$-linear, it is better
to define stratified bundles and morphisms in the relative version:
objects consist of
$$E'=\{E_0=E'_0, E'_1,
E'_2, \cdots, \sigma'_0, \sigma'_1,\ldots
\}=\{E'_i,\sigma'_i\}_{i\in\Bbb N}$$ where $E'_i$ is a bundle on the
$i$-th Frobenius twist $X^{(i)}$ of $X$, $$\sigma'_i:
F_{i,i+1}^*E_{i+1}'\to E_i'$$ is a $\sO_{X^{(i)}}$-linear
isomorphism, and $F_{i, i+1}: X^{(i)}\to X^{(i+1)}$ is the relative
Frobenius, the morphisms are the obvious ones. A rational point
$a\in X(k)$ yields a fiber functor $\omega_a: \textbf{str}(X) \to
\textbf{vec}_k,\quad E\mapsto a^*E_0$ with values in the category of
finite dimensional vector spaces. Thus $(\textbf{str}(X),\omega_a)$
is a Tannaka category (\cite[Section~2.2]{dS}), and one has an
equivalence of categories
\ga{2.1}{\textbf{str}(X)\xrightarrow{\omega_a \  \cong}\textbf{
rep}_k(\pi^{{\rm str}})} where
$$ \pi^{{\rm str}}={\rm Aut}^{\otimes} (\omega_a)$$
is the Tannaka group scheme, and $\textbf{rep}_k(\pi^{{\rm str}})$
is the category of finite dimensional $k$-representations of
$\pi^{{\rm str}}$. Let $\pi_1:=\pi_1^{{\rm \acute{e}t}}(X, a)$ be
the {\'e}tale fundamental group of $X$. In  \cite[Theorem~1.10]{Gi},
D. Gieseker proved the following theorem.

\begin{thm}\label{thm2.2}(Gieseker): Let $X$ be a smooth
projective variety over an algebraically closed field $k$, then
\begin{itemize}
\item[(i)] If every stratified bundle is trivial,
then $\pi_1$ is trivial.
\item[(ii)] If all the irreducible stratified bundles have rank $1$,
 then $[\pi_1,\pi_1]$ is a pro-$p$-group.
\item [(iii)] If every stratified bundle is a direct sum of stratified
line bundles, then $\pi_1$ is abelian with no $p$-power order
quotient.
\end{itemize}
\end{thm}
Then Gieseker conjectured that the converse of the above statements
might be true. In \cite{EM}, it is proven that the converse of
statement (i) is true. The aim of this section is to prove the
converse of (ii). We prove the converse of (iii) in Section 3.

The proof relies on the following theorem extracted from \cite{EM},
and which plays a similar r\^ole as the finite generation of the
topological fundamental group $\pi_1^{{\rm top}}$ in complex
geometry (see Remark~\ref{rmk3.12}).

 Fixing an ample line bundle $\sO_X(1)$ on $X$, recall
that $E$ is said to be $\mu$-stable  if for all coherent subsheaves
$U\subset E$ one has $\mu(U)<\mu(E)$, where $\mu(E)$ is the slope
which is defined to be the degree of $E$ with respect to $\sO_X(1)$
divided by the rank of $E$.

\begin{thm}\label{thm2.3} Let $X$ be a smooth
projective variety over an algebraically closed field $k$.  If
there is a stratified bundle $E=(E_n,\sigma_n)_{n\in \Bbb N}$ of
rank $r\ge 2$ on $X$, where $\{E_n\}_{n\in \Bbb N}$ are
$\mu$-stable, then there exists an irreducible
 representation $\rho:\pi_1\to GL(V)$ with
finite monodromy, where $V$ is a $r$-dimensional vector space over
${\bar{\F}_p}$.

\end{thm}
\begin{proof}
By the proof of \cite[Theorem~3.15]{EM}, there is a smooth
projective model $X_S\to S$ of $X \to \Spec k$, with model $a_S\in
X_S(S)$ of $a\in X(k)$, where $S$ is smooth affine over $\F_p$,
together with a quasi-projective  model $M_S\to S$ of the moduli of
stable vector bundles considered in \cite[Section~3]{EM},  there is
a closed point $u\to M_S$  of residue field $\F_q$, corresponding to
a bundle $\sE$ over $X_s$, stable over $X_{\bar s}=X_s\otimes
\bar{\mathbb{F}}_p$, where $s$ is the closed point $u$ viewed as a
closed point of $S$, such that $(F^m)^*\sE\cong \sE$, for some $m\in
\N\setminus \{0\}$, where $F$ is the Frobenius of $X_s/\F_q$. Thus,
as explained in \cite[Theorem~3.15]{EM}, $\sE$ trivializes over a
Lang torsor $Y\to X_s$ \cite[Satz~1.4]{LS}, thus its base change
$\sE_{\bar s}$ to $X_{\bar s}$ trivializes over $Y\times_s \bar s\to
X_{\bar s}$. This trivialization defines an irreducible
representation $ \rho':   \pi_1^{{\rm \acute{e}t}}(X_{\bar s},b) \to
GL( \sE_b)$, where $b=a_S\otimes \bar s$. As the specialization
homomorphism ${\rm sp}: \pi_1\to  \pi_1^{{\rm \acute{e}t}}(X_{\bar
s},b)$ is surjective (\cite[Expos\'e~X,~Th\'eor\`eme~3.8]{SGA1}),
the composite $\rho={\rm sp}\circ \rho'$ is   a representation of
$\pi_1$ with the same finite irreducible monodromy. This is a
solution to the problem.

\end{proof}

We now use the following two elementary lemmas.

\begin{lem}[See Chapter~8, Proposition~26 of \cite{Serre}]
\label{lem2.4} Let $G$ be a finite $p$-group and $\rho:G\to GL(V)$
be a representation on a vector space $V\neq 0$ over a field $k$ of
characteristic $p$. Then
$$V^G=\{\,v\in V\,|\, \rho(g)v=v,\,\, \forall\, g\in G\,\,\}\neq
0.$$
\end{lem}
\begin{lem}\label{lem2.5} Let $G$ be a finite group and $\rho:G\to GL(V)$ be an
irreducible representation on a finite dimensional vector space $V$
over an algebraically closed field $k$ of characteristic $p>0$. If
the commutator $[G,G]$ of $G$ is a $p$-group, then ${\rm dim}(V)=1$.
\end{lem}

\begin{proof} If $[G,G]$ is a $p$-group, by Lemma~\ref{lem2.4},
there exists a $0\neq v\in V$ such that
$\rho(g_1)\rho(g_2)v=\rho(g_2)\rho(g_1)v$ for any $g_1,\, g_2\in G$.
Since $V$ is an irreducible $G$-module, the sub-vector space spanned
by the orbit $\rho(G) v$ is $V$ itself. Thus for all $w\in V$, all
$g\in G$, there are $a_g(w)\in k$, such that
$$ w=\sum_{g\in G}a_g(w)\rho(g)v.$$ On the other hand, for all
$g_1,g_2,g\in G$, one has $g_1 g_2=g_2g_1 c$ for some $c\in [G,G]$,
thus $g_1g_2
 g=g_2 g_1 c  g=g_2 g_1 g
c'$, for some $c'\in [G,G]$. We conclude that
$\rho(g_1)\rho(g_2)w=\rho(g_2)\rho(g_1)w$ for any $g_1$, $g_2\,\in
G$, that is  $\rho(G)$ is abelian. As $\rho$ is irreducible, $V$
must have dimension $1$.

\end{proof}

\begin{thm}\label{thm2.6} Let $X$ be a smooth
projective variety over an algebraically closed field $k$. Then all
irreducible stratified bundles on $X$ have rank $1$ if and only if
the commutator $[\pi_1,\pi_1]$ of $\pi_1$ is a pro-$p$-group.
\end{thm}

\begin{proof} One direction is the (ii) in Theorem~\ref{thm2.2}. We prove
the converse. Assume that $[\pi_1,\pi_1]$ is a pro-$p$-group. Let
$E=(E_n,\sigma_n)_{n\in \Bbb N}$ be an irreducible stratified bundle
of rank $r\ge 1$ on $X$. Assume $r\ge 2$. Then by \cite[Proposition
2.3]{EM}, there is a  $n_0 \in \N$ such that
 the stratified bundle $E(n_0):=(E_n,
\sigma_n)_{n\ge n_0} $ is a successive extension of stratified
bundles $U=(U_n, \tau_n)$ with underlying bundles $U_n$ being
$\mu$-stable. But $E$ being irreducible implies that  $E(n_0)$ is
irreducible as well. Hence all the $E_n$ are $\mu$-stable for $n\ge
n_0$. By Theorem~\ref{thm2.3},  there is an irreducible
representation $\rho: \pi_1\to GL(V)$ of dimension $r\ge 2$ with
finite monodromy over an algebraically closed field of
characteristic $p>0$. By Lemma~\ref{lem2.5}, this is impossible.
Thus $r=1$.
\end{proof}

\section{Extensions of stratified line bundles}

In this section, we prove the following theorem.

\begin{thm}\label{thm3.1} Let $X$ be a smooth projective connected
variety over an algebraically closed field $k$ of characteristic
$p>0$. If the {\'e}tale fundamental group $\pi_1$ of $X$ is abelian
and has no non-trivial $p$-power order quotient, then any extension
$$0\to \mathbb{L}\to \mathbb{V}\to \mathbb{L}'\to 0$$
in $\rm{\bf{str}}(X)$ is split when $\mathbb{L}$ and $\mathbb{L}'$
are rank $1$ objects.
\end{thm}

We start the proof by the following proposition, which can be
considered as a generalization of \cite[Proposition 2.4]{EM}.

\begin{prop}\label{prop3.2}
Let $L$ be a line bundle on $X$ such that $(F_X^*)^aL$ is isomorphic
to $L$, for some non zero natural number $a$, where $F_X: X\to X$ is
the absolute Frobenius map. Then
$$(F_X^*)^a: H^1(X,L)\to H^1(X,L)$$ is nilpotent if $\pi_1$
is abelian without non-trivial $p$-power quotient.
\end{prop}

\begin{proof} We choose an isomorphism  $(F_X^*)^aL\cong L$, or
equivalently an isomorphism $L^{p^a-1}\cong \sO_X$. This define a Kummer
cover $\phi: Y\to X$ with Galois group $H=\Z/(p^{a}-1)$  such that
$$\phi^*L=\sO_Y.$$
One has an exact sequence  (recall that $\pi_1:=\pi_1^{{\rm
\acute{e}t}}(X,\bar a)$)
$$1\to  \pi_1^{{\rm \acute{e}t}(Y,\bar b)}\to  \pi_1\to
H\to 1,$$ where $\bar b$ is a geometric point of $Y$ above the
geometric point $\bar a$ of $X$.

\begin{claim}\label{claim3.3}   $\pi_1^{{\rm \acute{e}t}}(Y, \bar b)$ is
abelian without $p$-power quotient.
\end{claim}
\begin{proof}
 Since  $\pi_1$ is commutative, the
kernel of any quotient map $$ \pi_1^{{\rm \acute{e}t}}(Y,\bar b)
\surj K$$ is normal in  $\pi_1$. Taking for $K$ a finite $p$-group
defines the push-out exact sequence $$1\to K\to G\to H\to 1$$ where
$G$ is a quotient of  $\pi_1$, thus is commutative and has no
$p$-power quotient. Since $K$ splits in $G$, $K=\{1\}$.
\end{proof}
 Then, by
\cite[Proposition 2.4]{EM}, there is an
integer $N>0$ such that $$(F_Y^*)^N: H^1(Y,\sO_Y)\to H^1(Y,\sO_Y)$$
is a zero map, which implies that
$$\phi^*\cdot (F_X^*)^{Na}=(F_Y^*)^{Na}\cdot\phi^*: \,\,H^1(X,L)\to
H^1(Y,\sO_Y)$$ is a zero map. But $\phi^*:H^1(X,L)\hookrightarrow
H^1(Y,\phi^*L)=H^1(Y,\sO_Y)$ is injective, thus $(F_X^*)^{Na}:
H^1(X,L)\to H^1(X,L)$ is a zero map.

\end{proof}



\begin{proof}[Proof of Theorem~\ref{thm3.1}] By twisting with
$(\mathbb{L}')^{-1}$, we may assume that $\mathbb{L}'= \mathbb{I}$,
the trivial object in $\textbf{str}(X)$. We prove
Theorem~\ref{thm3.1} by contradiction. If there exists a nontrivial
extension \ga{3.1} {0\to \mathbb{L}\to\mathbb{V}\to \mathbb{I}\to
0}in $\textbf{str}(X)$, then, by definition,  there is a set
$$\mathbb{E}_X=\{\,\,\,(L_i\hookrightarrow
V_i\twoheadrightarrow\sO_X)\,\,\}_{i\in \mathbb{N}}$$ of isomorphism
classes of non-trivial extensions on $X$ such that
$$(L_i\hookrightarrow V_i\twoheadrightarrow\sO_X)\cong(F_X^*L_{i+1}\hookrightarrow F_X^*V_{i+1}\twoheadrightarrow
F^*_X\sO_X)$$ for any $i\in \mathbb{N}$ {\it modulo scale}. This
means that there is are $\lambda_i \in k^\times=\Gamma(X,
\sO^\times_X)$ such that the push down of $(L_i\hookrightarrow
V_i\twoheadrightarrow\sO_X)$ by $L_i\xrightarrow{\lambda_i} L_i$ is
isomorphic as an extension to $(F_X^*L_{i+1}\hookrightarrow
F^*_XV_{i+1}\twoheadrightarrow\sO_X)$. In particular, $L_i$ is
isomorphic as a line bundle to $F_X^*(L_{i+1})$.

 If $\mathbb{E}_X$ is a
finite set, then there are only finitely many non-isomorphic $L_i$,
thus there is a $a\in \N\setminus \{0\}$ such that
$(F^*_X)^aL_i\cong L_i$  for all $i$. We obtain a contradiction by
Proposition~\ref{prop3.2}.

If $\mathbb{E}_X$ is a infinite set, by Theorem~\ref{thm3.4} below,
there is a nontrivial extension
\ga{3.2}{0\to L\to V\to \sO_{X_{\bar
s}}\to 0} on a good reduction $X_{\bar s}$ of $X$ (thus over
$\overline{\mathbb{F}}_p$) such that for an integer $a>0$
$$0\to (F^*_{X_{\bar s}})^aL\to (F^*_{X_{\bar s}})^aV\to (F^*_{X_{\bar s}})^a\sO_{X_{\bar s}}\to 0$$
is isomorphic to the extension \eqref{3.2}  modulo scale. Then, by
using Proposition~\ref{prop3.2} to $X_{\bar s}$, we obtain again a
contradiction since $\pi_1^{{\rm \acute{e}t}}(X_{\bar s}, a_{\bar
s})$, which is a quotient of  $\pi_1$ via the specialization map, is
also abelian without non-trivial $p$-power quotient. Here $a_{\bar
s}$ is a specialization of the geometric point $\bar a$ used to
define $\pi_1$.
\end{proof}

\begin{thm}\label{thm3.4} Let $X$ be a smooth projective variety
over an algebraically closed field $k$.  If there exists an infinite
set $\mathbb{E}_X$ of equivalent classes of nontrivial extensions
$$(L_i\hookrightarrow V_i\twoheadrightarrow\sO_X),\quad i\in\mathbb{N}$$
satisfying $(L_i\hookrightarrow V_i\twoheadrightarrow\sO_X)\cong
(F^*_XL_{i+1}\hookrightarrow F^*_XV_{i+1}\twoheadrightarrow
F^*_X\sO_X)$ for any $i\in\mathbb{N}$,  modulo scale,  then there
exists a nontrivial extension $$0\to L\to V\to \sO_{X_{\bar s}}\to
0$$ on a good reduction $X_{\bar s}$ of $X$ (over $
\overline{\mathbb{F}}_p$) such that for some $a>0$
$$\left((F^*_{X_{\bar s}})^aL\hookrightarrow (F^*_{X_{\bar s}})^aV\twoheadrightarrow (F^*_{X_{\bar s}})^*\sO_{X_{\bar
s}}\right)\cong (L\hookrightarrow V\twoheadrightarrow\sO_{X_{\bar
s}})$$ modulo scale.
\end{thm}

We now prove Theorem~\ref{thm3.4}.  Let $X_S\to S$ be a smooth
projective model of $X/k$, endowed with a section, where $S$ is a
smooth affine irreducible variety over $\mathbb{F}_{q}$ with $H^1(X,
\sO_{X_S})=H^0(S, \sO_S)$. So $X_S\to S$ has smooth projective
geometrically irreducible fibers. Then we have
\begin{prop}\label{prop3.5} There exists a reduced $S$-scheme $M\to
S$ of finite type and a rational map $f: M\dashrightarrow M$ over
$S$ such that
\begin{itemize}
\item [(1)]
For any field extension $K\supset \F_q$, the set $M(K)=\{e\}$ of
$K$-valued points consists of isomorphism classes of non-trivial
extensions  $L\hookrightarrow V\twoheadrightarrow\sO_{X_{s}}$ modulo
scale on $X_{ s}$, where $X_{ s}$ is the fiber of $X_S\to S$ at the
image $s\in S(K)$ of $e \in M(K)$.
\item [(2)] The set  $ \mathbb{E}_X$ of Theorem~\ref{thm3.4} minus
finitely many elements lies in $ M(k)$.
\item [(3)]  The rational map  $f:M\dashrightarrow M$ is
well-defined at $e=(L\hookrightarrow V\twoheadrightarrow\sO_{X_{
s_0}})\in M(K)$  if and only if the pull-back extension
$$0\to F^*_{X_{s}}L\to F^*_{X_{s}}V \to F^*_{X_s}
\sO_{X_s}\to 0$$ under the absolute Frobenius $F_{X_{s}}$ does not
split, in which case
$$f(e)=(F^*_{X_{s}}L\hookrightarrow F^*_{X_{s}}V \twoheadrightarrow
F^*_{X_{s}}\sO_{X_{ s}})  \in M(K).$$
\end{itemize}
\end{prop}

\begin{proof}
The line bundles $\{\,L_i\,\}_{i\in \mathbb{N}}$ occurring in
$\mathbb{E}_X$ must satisfy $$L_i\cong L_{i+1}^p,\quad
\forall\,\,i\in\mathbb{N},$$ so they are infinitely $p$-divisible,
thus lie in $\Pic^\tau(X)(k)$, the group of numerically trivial line
bundles over $k$ (\cite[9.6]{Kl}).

We first assume $\mathbb{L}=\mathbb{I}$. We shrink $S$ so that
$H^1(\sO_{X_S})$ is locally free and commutes with base change. Let
$H^1(\sO_{X_S})$ denote the first direct image of $\sO_{X_S}$ under
$X_S\to S$, then
$$\pi_S: M=\mathbb{P}(H^1(\sO_{X_S})^{\vee})\to S$$
is the moduli space of isomorphism classes of non-trivial extensions
$ \sO_{X_{\bar s}}\hookrightarrow V\twoheadrightarrow\sO_{X_{\bar
s}}$. In the diagram
$$\xymatrix{\ar@/^20pt/[rr]^{F_{X_S}} X_S \ar[r]^{F} \ar[dr] & X_S'\ar[r]\ar[d]^{}
& X_S\ar[d]\\
 & S\ar[r]^{F_S}& S,} $$
$F_S$, $F_{X_S}$ are the absolute Frobenius and $F:X_S\to X_S'$ is
the relative Frobenius. Then the "Frobenius pullback" induces
\ga{3.3}{F^*: F_S^*H^1(\sO_{X_S})=H^1(\sO_{X_S'})\to H^1(\sO_{X_S})}
which is $\sO_S$-linear and commutes with base change. Let
$$\pi_S^*H^1(\sO_{X_S})^{\vee}\to\sO_M(1)\to 0$$
be the universal rank one quotient, which, pulled back by the
absolute Frobenius $F_M:M\to M$, defines the $\sO_M$-linear map
\ga{3.4}{\pi_S^*F_S^*H^1(\sO_{X_S})^{\vee}\to \sO_M(p)\to 0} (since
$F_M^*\pi_S^*=\pi_S^*F_S^*$ and $F_M^*\sO_M(1):=\sO_M(p)$).
Combining the dual
$$\pi_S^*({F^*}^{\vee}): \pi_S^*H^1(\sO_{X_S})^{\vee}\to
\pi_S^*F_S^*H^1(\sO_{X_S})^{\vee}$$ of \eqref{3.3}
 with \eqref{3.4}
 defines the $\sO_M$-linear map
$$ \pi_S^*H^1(\sO_{X_S})^{\vee} \to
\sO_M(p). $$ By the universal property, this is the same as a
rational map
$$f: M\dashrightarrow M$$
over $S$.  To see that $f$ satisfies (3), let $e\to M$ be a
point with image $s\to S$ (under $M\xrightarrow{\pi_S}S$) , which
corresponds a nontrivial extension
$$e=(\sO_{X_s}\hookrightarrow V\twoheadrightarrow
\sO_{X_s})$$ where $X_s$ is the fiber of $X_S\to S$ at $s\to S$.
Then $f$ is well-defined at $e\to M$ if and only if the
$\sO_M$-linear map $$\pi_S^*H^1(\sO_{X_S})^{\vee} \to \sO_M(p)$$ at
$e\to M$ is surjective. By definition, this is equivalent to saying
that the $\sO_S$-linear map \eqref{3.3} at $s\to S$, which is
$$F^*: F_{k(s)}^*H^1(\sO_{X_s})=H^1(\sO_{X'_s})\to H^1(\sO_{X_s}),$$
is not trivial at $F^*_{k(s)}([e])\in H^1(\sO_{X_s'})$ (i.e.
$F^*(F^*_{k(s)}([e]))\neq 0$), where $$[e]\in H^1(\sO_{X_s})$$
denotes the corresponding class of $e=(\sO_{X_s}\hookrightarrow
V\twoheadrightarrow \sO_{X_s})$. Since
$$F_{X_s}^*: H^1(\sO_{X_s})\xrightarrow{F_{k(s)}^*}
H^1(\sO_{X_s'})\xrightarrow{F^*} H^1(\sO_{X_s}),$$ we have
$F^*_{X_s}([e])=F^*(F^*_{k(s)}([e]))$. Thus $f$ is well-defined at
$$e=(\sO_{X_s}\hookrightarrow V\twoheadrightarrow \sO_{X_s})\in M$$
if and only if $(F^*_{X_s}\sO_{X_s}\hookrightarrow
F^*_{X_s}V\twoheadrightarrow F^*_{X_s}\sO_{X_s})$ is not splitting,
and $$f(e)=(F^*_{X_s}\sO_{X_s}\hookrightarrow
F^*_{X_s}V\twoheadrightarrow F^*_{X_s}\sO_{X_s}).$$ The objects
$(\sO_X\hookrightarrow V_i\twoheadrightarrow\sO_X)\in \E_X$
($i\in\mathbb{N}$) define points $e_i\to M$ over $s=\Spec(k)\to S$,
and $f$ is well-defined at $e_i$ with $f(e_i)=e_{i-1}$ for $i>1$.
This finishes the proof if $\mathbb{L}=\mathbb{I}$.

If $\mathbb{L}\neq \mathbb{I}$, after removing a finite number of
elements in the set  $\mathbb{E}_X$, we can assume that all line
bundles $\{\,L_i\,\}_{i\in \mathbb{N}}$ occurring in $\mathbb{E}_X$
satisfy
$$ H^0(X,L_i)=0,\,\,H^1(L_i)\neq 0\quad i\in \mathbb{N}.$$
Let ${\rm Pic}^{\tau}_{X_S}\to S$ be the torsion component of the
identity of the Picard scheme \cite[9.6]{Kl}. Let $\sL$ be the
universal line bundle on $X_S\times_S {\rm Pic}^{\tau}_{X_S}$.  We
define
$$\mathcal{N}_0=\{\,\, t\in {\rm Pic}^{\tau}_{X_S}\,\,|\,\,
H^1(\sL_t)\neq 0\,\,\}$$ with its reduced structure. By
semi-continuity of cohomology, it is a closed sub-scheme of ${\rm
Pic}^{\tau}_{X_S}$. If $\sN_i$ is defined, let
$$\sN_{i+1}=\{\,\, t\in \sN_i\,\, |\,\,H^1(\sL^{p^{i+1}}_t)\neq 0\,\,\}.$$
By the noetherian property, $\mathcal{N}_0\supseteq
\mathcal{N}_1\supseteq\cdots\mathcal{N}_i\supseteq\mathcal{N}_{i+1}\supseteq\cdots$
terminates. Then,  there is a $k_0\ge 0$ such that
$$\mathcal{N}_i=\mathcal{N}_{k_0},\quad\, \forall\,\, i\ge k_0.$$
The line bundles $\{\,L_{i+p^{k^0}}\,\}_{i\in \mathbb{N}}$ occurring
in $\mathbb{E}_X(p^{k_0})$, where
$$\mathbb{E}_X(p^{k_0}):=\{\,\,(L_{i+p^{k_0}}\hookrightarrow
V_{i+p^{k_0}}\twoheadrightarrow\sO_X)\in \mathbb{E}_X\,\,\}_{i\in
\mathbb{N}},$$ are $k$-points of $\mathcal{N}_{k_0}=\mathcal{N}_i$
for all $i\ge k_0$.

We define $T\subset\mathcal{N}_{k_0}\to S$ with its reduced
structure to be the sub-scheme
$$T=\{\,\,t\in \mathcal{N}_{k_0}\,\,|\,\,\, H^0(\sL_t)=0\,\,\}.$$
It is open in $\mathcal{N}_{k_0}$. Let $\sL$ be the restriction of
universal line bundle on $X_S\times_S T$ (thus, for any $t\in T$,
$H^1(\sL^{p^i}_t)\neq 0$ for $i\ge 0$). Unfortunately,
$R^1{p_T}_*(\sL)$ is neither commuting with base change nor locally
free in general, where $p_T: X_S\times_S T\to T $ is the projection.
However,
$$\sE=R^{n-1}{p_T}_*(\sL^{\vee}\otimes p^*_{X_S}\omega_{X_S/S}),\quad n=\dim(X)$$
may not be locally free, but does commute with base change since
$$H^n(\sL_t^{\vee}\otimes
\omega_{X_S\times\{t\}})=H^0(\sL_t)^{\vee}=0,\quad \forall\,\,t\in
T.$$ There exists a quotient scheme $\pi:M=\mathbb{P}(\sE)\to T$
together with $$\pi^*\sE\to\sO_M(1)\to 0,$$ which represents the
functor that sends a $T$-scheme $p_W:W\to T$ to the set of
isomorphic classes of quotients $p_W^*\sE\to Q\to 0$, where $Q$ is a
line bundle on $W$ (\cite[Thm~2.2.4]{HL}).  By definition,
$M\xrightarrow{\pi}T\to S$ satisfies the properties (1) and (2) in
the proposition (replacing $M$ by $M_{{\rm red}}$ if necessary). The
rest of the proof is devoted to the construction of a rational map
$$f:M\dashrightarrow M$$
over $S$, which satisfies the property (3) in the proposition.

Let $F:X_S\times_S T\to (X_S\times_S T)'$ denote the relative
Frobenius morphism over $T$.
We write
\ga{3.5}{\xymatrix{\ar@/^20pt/[rr]^{F_{X_S\times_S T}} X_S\times_S
T\ar[r]^F \ar[dr]_{p_T} & (X_S\times_S T)'\ar[r]\ar[d]^{p'_T}
& X_S\times_S T\ar[d]^{p_T}\\
 & T\ar[r]^{F_T}& T.} }
One has $(X_S\times_S T)'=X'_S\times_S T$ thus we can project to
$X'_S$, further to $X_S$. Let $\sL'$ be the pullback of $\sL$ under
the projection
$$(X_S\times_S T)'\xrightarrow{p_1} X_S\times_S T$$ and $\omega$ (resp.
$\omega'$) be the relative canonical line bundle of $X_S\times_S
T\xrightarrow{p_T} T$ (resp. $(X_S\times_S T)'\xrightarrow{p'_T}
T$). The line bundle $F^*\sL'$ on $X_S\times_ST$ defines a
$S$-morphism $\nu:T\to{\rm Pic}^{\tau}_{X_S}$. Define
$$T^0=\{\,t\in
T\,|\,H^0(\sL_t^p)= 0\,\}\subset T$$ with its reduced scheme
structure. It is open in $T$.  By definition of $T$, $\nu$
restricted to $T^0$ factors through $T$, thus induces
$$\nu: T^0\to T$$ such that
$$(1\times \nu)^*\sL\otimes p_{T^0}^*\eta= F^*\sL'$$ on $X_S\times_S T^0$, with
$X_S\times_S T^0\xrightarrow{p_{T^0}} T^0,$ and where
 $\eta$ is a line bundle on
$T^0$.  We abuse notations and still write $\omega$ for the relative
dualizing sheaf $\omega_{X_S\times_S T^0/T^0}$ of $p_{T^0}$,
$\omega'$ for the relative dualizing sheaf $\omega_{X'_S\times_S
T^0/T^0}$ of $p'_{T^0`}: X'_S\times_S T^0\to T^0$ . Since
$\omega=p_{X_S}^*\omega_{X_S/S}=(1\times\nu)^*\omega_{X_S\times_S
T/T}$, we have
$$R^{n-1}{p_{T^0}}_*(\omega\otimes
{F^*\sL'}^{\vee})=\eta^{-1}\otimes\nu^*\sE,\quad
R^{n-1}{p'_{T^0}}_*(\omega'\otimes{\sL'}^{\vee})=F_{T^0}^*\sE, $$
and those identities commute with base change. By Lemma~\ref{lem3.6}, we have
$$\phi:\eta^{-1}\otimes\nu^*\sE\to
F_{T^0}^*\sE.$$ Together with the universal quotient
$\pi^*\sE\to\sO_{M_0}(1)\to 0$, restricted to $M_0:=M\times_S T^0$,
where $\pi:M_0\to T^0$, $T^0\to S$ factors through the composition
of the open embedding $T^0\subset T$ with the map $T\to S$, this
induces
$$ \Phi: \pi^*\nu^*\sE\xrightarrow{\phi}
\pi^*\eta\otimes\pi^*F_{T^0}^*\sE=\pi^*\eta\otimes
F_{M_0}^*\pi^*\sE\to \pi^*\eta\otimes \sO_{M_0}(p).$$  Here
$F_{T^0}$, $F_{M_0}$ are the absolute Frobenius morphisms satisfying
$$\begin{CD}
M_0&@>F_{M_0}>>& M_0 \\
@VV{\pi}V &&@VV{\pi}V\\
T^0&@>F_{T^0}>>&T^0
\end{CD}$$

Let $M^0\subset M$ be the reduced open set consisting of points
$q\in M$ such that $\Phi$ is surjective at $q\in M$ (which implies
$M^0\subset M_0$). Then there exists a unique morphism $$f: M^0\to
M$$ corresponding via the universal property to $f^*(\pi^*\sE\to
\sO_M(1)\to 0)=(\pi^*\nu^*\sE\to\pi^*\eta\otimes \sO_{M^0}(p)\to
0)$. By definition, one has the factorization
$$\begin{CD}
M^0&@>f>>& M \\
@VV{\pi}V &&@VV{\pi}V\\
T^0&@>\nu>>&T.
\end{CD}$$
Then Lemma~\ref{lem3.6} below implies that the rational map
$$f: M\dashrightarrow M$$ satisfies the requirement (3) in the
proposition.
\end{proof}

For any point $t\to T$, let $s\to S$ be its image under $T\to S$,
and $X_s$ be the fiber of $X_S\to S$ at $s\to S$. Then the diagram
\eqref{3.5} specializes to
$$\xymatrix{\ar@/^20pt/[rr]^{F_{X_s\times t}} X_s\times
t\ar[r]^{F_t} \ar[dr] & (X_s\times t)'\ar[r]\ar[d]^{}
& X_s\times t\ar[d]\\
 & t\ar[r]^{}& t} $$
and $F_t$ induces the $k(t)$-linear map $F^*_t: H^1(\sL'_t)\to
H^1((F^*\sL')_t)$, which induces a $k(t)$-linear map
$$(F^*_t)^{\vee}: H^1((F^*\sL')_t)^{\vee}\to H^1(\sL'_t)^{\vee}.$$
By Serre duality, we have the induced $k(t)$-linear map
\ga{3.6}{\phi_t: H^{n-1}((\omega\otimes {F^*\sL'}^{\vee})_t)\to
H^{n-1}((\omega'\otimes{\sL'}^{\vee})_t).} Serre duality allows to
make \eqref{3.6} in families:

\begin{lem}\label{lem3.6}There exists a homomorphism
$$\phi:R^{n-1}{p_T}_*(\omega\otimes {F^*\sL'}^{\vee})\to
R^{n-1}{p'_T}_*(\omega'\otimes{\sL'}^{\vee})$$ such that
$$\begin{CD}
R^{n-1}{p_T}_*(\omega\otimes {F^*\sL'}^{\vee})\otimes k(t)&@>\tau>>& H^{n-1}((\omega\otimes {F^*\sL'}^{\vee})_t) \\
@VV{\phi\otimes k(t)}V &&@VV{\phi_t}V\\
R^{n-1}{p'_T}_*(\omega'\otimes{\sL'}^{\vee})\otimes
k(t)&@>\tau>>&H^{n-1}((\omega'\otimes{\sL'}^{\vee})_t)
\end{CD}$$
is commutative for any $t\to T$, where $\tau$ is the canonical base
change isomorphism, and $\phi_t$ is the homomorphism in \eqref{3.6}.
\end{lem}

\begin{proof}  Let
$\omega$ (resp. $\omega'$) be the relative canonical line bundle of
$$X_T:=X_S\times_S T\xrightarrow{p_T}T$$ (resp. $X_T':=(X_S\times_S
T)'\xrightarrow{p'_T}T$).  Let $F:X_T\to X_T'$ be the relative
Frobenius morphism in \eqref{3.5}, which is flat and Cohen-Macaulay
as $X_S\to S$ is smooth. The injection $\sL'\hookrightarrow
F_*F^*\sL'$ of vector bundles induces the surjection of vector
bundles \ga{3.7} {\sH om(F_*F^*\sL',\omega')\twoheadrightarrow \sH
om(\sL',\omega').} Duality theory for the  Cohen-Macaulay map
$F:X_T\to X_T'$ implies an isomorphism \ga{3.8}{F_*\sH
om(F^*\sL',\omega)\cong \sH om(F_*F^*\sL',\omega').} Equations
\eqref{3.7} and \eqref{3.8} imply that one has a surjection \ga{3.9}
{F_*\sH om(F^*\sL',\omega)\twoheadrightarrow \sH om(\sL',\omega')}
of vector bundles on $X'_S$. Since $F$ is a finite morphism, taking
$R^{n-1}{p'_T}_*(\eqref{3.9})$ induces
\ga{3.10}{\phi:R^{n-1}{p_T}_*\sH om(F^*\sL',\omega)\to
R^{n-1}{p'_T}_*\sH om(\sL',\omega').} For any $t\to T$, let $X_t$ be
the fiber of $X_T\to T$, and $F_t:X_t\to X_t'$ be the relative
Frobenius. Then $\eqref{3.10}\otimes k(t)$ induces (through $\tau$)
$$\phi\otimes k(t):H^{n-1}(X_t,\sH om(F^*_t\sL'_t,\omega_t))\to
H^{n-1}(X'_t,\sH om(\sL'_t,\omega'_t))$$ which is induced by
$\varphi:=\eqref{3.9}\otimes k(t)$.  Then the surjection of vector
bundles on $X'_t$ \ga{3.11}{\varphi:(F_t)_*\sH
om(F_t^*\sL'_t,\omega_t)\twoheadrightarrow \sH om(\sL_t',\omega_t')}
is dual, by taking $\sH om(\cdot,\,\omega'_t)$, to the canonical
injection \ga{3.12} {\sL_t'\hookrightarrow (F_t)_*F^*_t\sL'_t\,} of
vector bundles on $X'_t$. As the dual $H^{n-1}(\varphi)^{\vee}$ of
$H^{n-1}(\varphi)=\phi\otimes k(t)$ is the $k(t)$-linear map
$$F_t^*: H^1(X'_t,\sL'_t)\to H^1(X_t, F^*_t\sL_t')$$
induced by \eqref{3.12}, this finishes the proof.
\end{proof}

To prove Theorem~\ref{thm3.4}, the key tool is a theorem of
Hrushovski. In fact, we only need a special case of his theorem
\cite[Corollary~1.2]{H}.

\begin{thm}[Hrushovski, { \cite[Corollary~1.2]{H}}]\label{thm3.7}
Let $Y$ be an affine variety over $\mathbb{F}_q$, and let
$\Gamma\subset (Y\times_{\F_q} Y)\otimes_{\F_q} \bar \F_q$ be an
irreducible subvariety over $\overline{\mathbb{F}}_q$. Assume the
two projections $\Gamma\to X$ are dominant. Then, for any closed
subvariety $W\varsubsetneq Y$, there exists $x\in
Y(\overline{\mathbb{F}}_q)$ such that $(x, x^{q^m})\in \Gamma$ and
$x\notin W$ for large enough natural number $m$.
\end{thm}

\begin{rmk} \label{rmk3.8}
 Writing $Y\subset \mathbb{A}^r$, then
 $x=(x_1,...,x_r)\in (\overline{\mathbb{F}}_q)^{r}=\mathbb{A}^r(
\overline{\mathbb{F}}_q)$
 and $x^{q^m}:=(x_1^{q^m},...,x_r^{q^m})\in  \mathbb{A}^r(
\overline{\mathbb{F}}_q) $. In our application,
 $\Gamma$ is the Zariski closure of the graph of a dominant rational map
 $f:Y\dashrightarrow Y$ and $W$ is the locus where $f$ is not
 well-defined.
\end{rmk}
\begin{proof}[Proof of Theorem~\ref{thm3.4}]
 (Compare with  \cite[Thm~3.14]{EM}.) Let
 $$f: M \dashrightarrow M$$ be as in  Proposition~\ref{prop3.5}. So $M\to S$ is a $S$-reduced scheme of
finite type, where $S$ is smooth affine over $\F_q$, the
normalization of $\F_p$ in $H^0(S, \sO_S)$.
 Let $M_k$ be the general fiber of $M\to S$
at $\Spec(k)\to S$.  We define  $Z_k$ to be the intersection of the
Zariski closures $\overline{\E_X(m)} $ of the $\E_X(m)\subset M_k$,
where $\mathbb{E}_X(m)=\{\,(L_{i+m}\hookrightarrow
V_{i+m}\twoheadrightarrow \sO_X)\in
\mathbb{E}_X\,\}_{i\in\mathbb{N}}$. By the noetherian property,
there exists a $m_0>0$ such that
$$Z_k=\overline{\mathbb{E}_X(m)}\subset  M_k(k)=M(k)$$ for any $m\ge
m_0$. As $\E_X$ is assumed to be infinite,  all components of $Z_k$
have dimension $\ge 1$. Indeed, if it had a component of dimension
$0$, then there would be a $m_1>m_0$ such that this $0$-dimensional
component would not lie on $\E_X(m_1)$, a contradiction. The fields
of constants of the irreducible components of $Z_k$ lies between
$\F_q(S)$ and $k$. Let $S'\to S$ be affine with $S'$ smooth affine
irreducible such that all components of $Z_k$ are defined over $S'$,
yielding models of them over $S'$ which are geometrically
irreducible over $\F_q$, and a model $Z\to S'$ of $Z_k$. We replace
now $X_S\to S$ and $M\to S$ by $X_{S'}$ and the corresponding
$M_{S'}\to S'$. We abuse notations, set $S=S'$, and have $Z\subset
M\to S$.

 Let $M^0\subset M$ be the largest open subset where $f$ is
well-defined. Notice that $$\mathbb{E}_X(1)\subset M^0_k$$ and
$f(\mathbb{E}_X(m))=\mathbb{E}_X(m-1)$ for any $m\ge 1$. Thus
$f(Z_k\cap M^0_k)$ contains a dense subset of $Z_k$. On the other
hand, $$f(Z_k\cap M^0_k)\subset Z_k,$$ thus $f:M\dashrightarrow M$
induces a dominant rational map $f: Z_k\dashrightarrow Z_k, $ and
thus a rational dominant map $$f: Z\dashrightarrow Z.$$ We consider
one irreducible component of $Z_k$, and its model $Z^{\rm irr}\to
S$, which is a geometrically irreducible component of $Z$ over
$\F_q$.  There is an integer $a_1>0$ such that $f^{a_1}: Z^{\rm
irr}\dashrightarrow Z^{\rm irr}$ is a dominant rational map over a
finite field $\mathbb{F}_q$.

To prove that $Z^{\rm irr}$ contains a periodic point of $f^{a_1}$,
without loss of generality, we can assume that $Y:=Z^{\rm
irr}\subset \mathbb{A}^r$ is an affine variety and $f^{a_1}:
Y\dashrightarrow Y$ is defined by rational functions $f_1, ..., f_r
\in \mathbb{F}_q(Y)$. Let
$$ \Gamma\subset Y\times_{\F_q} Y$$
be the Zariski closure of the graph of $f^{a_1}$, which is
irreducible over $\overline{\mathbb{F}}_q$  as $Y$ is by assumption.
The two projections $\Gamma\to Y$ are dominant since
$f^{a_1}:Y\dashrightarrow Y$ is dominant. By Theorem~\ref{thm3.7},
 there exists a point
$x\in Y(\overline{\mathbb{F}}_q)$ such that $f$ is well-defined at
$x=(x_1,...,x_r)$ and
$f^{a_1}(x)=(x_1^{q^m},...,x_r^{q^m}):=x^{q^m}$ for some large $m$.
Recall that $f_i\in \mathbb{F}_q(Y)$ ($i=1,..., r$) have
coefficients in $\mathbb{F}_q$, we have
$$\aligned f^{a_1}(f^{a_1}(x))&=(f_1(x_1^{q^m},...,x_r^{q^m}),...\,,f_1(x_1^{q^m},...,x_r^{q^m}))\\&=
(f_1(x_1,...,x_r)^{q^m},...\,,f_1(x_1,...,x_r)^{q^m})\\&=
(x_1^{q^{2m}},...\,,x_r^{q^{2m}})=x^{q^{2m}}.\endaligned$$ Thus
$f^{a_2a_1}(x)=x^{q^{a_2m}}=x$ when $a_2$ is large enough. The point
$$x\in Y(\overline{\mathbb{F}}_q)$$ determines a nontrivial extension
$0\to L\to V\to \sO_{X_{\bar s}}\to 0$ on a good reduction $X_{\bar
s}$ of $X$ (over $\overline{\mathbb{F}}_q$). That $f^a(x)=x$,
$a=a_1a_2$, means by Proposition~\ref{prop3.5} (3)
$$\left((F^*_{X_{\bar s}})^aL\hookrightarrow (F^*_{X_{\bar s}})^aV\twoheadrightarrow (F^*_{X_{\bar s}})^*\sO_{X_{\bar
s}}\right)\cong (L\hookrightarrow V\twoheadrightarrow\sO_{X_{\bar
s}})$$ up to scale. This finishes the proof of Theorem~\ref{thm3.4}
and thus of Theorem~\ref{thm3.1}.
\end{proof}

\begin{thm}  \label{thm3.9} Let $X$ be a smooth projective variety over
 an algebraically closed field $k$ of characteristic $p>0$,
then:
\begin{itemize}
\item[(i)] Every stratified bundle on $X$ is trivial if and only if $\pi_1$ is
trivial.
\item[(ii)] All the irreducible stratified bundles have rank $1$
if and only if $[\pi_1,\pi_1]$ is a pro-$p$-group.
\item [(iii)] Every stratified bundle is a direct sum of stratified
line bundles, that is $\mathbf{str}(X)$ is a semi-simple category with
irreducible objects of rank $1$,  if and only if $\pi_1$ is abelian with no
non-trivial
$p$-power quotient.
\end{itemize}
\end{thm}

\begin{proof} (i) is the main theorem of \cite{EM}. (ii) is
Theorem~\ref{thm2.6}.
To show (iii), assume that $\pi_1$ is
abelian with no $p$-power order quotient. Let
$E=(E_n,\sigma_n)_{n\in \Bbb N}$ be a stratified bundle on $X$. By
(ii), any irreducible stratified bundle has rank $1$. Thus there is
a filtration $0=E^0\subset E^1\subset \ldots \subset E^r=E$ in
$\mathbf{str}(X)$ such that $L^v=E^v/E^{v-1}$, $1\le v\le r$, are
rank one stratified bundles. Then, by Theorem~\ref{thm3.1}, the
filtration splits and $E$ is a direct sum of rank $1$ objects.
\end{proof}

We now   comment on analogs of (ii) and (iii) in complex geometry.
For this reason, we include here the following lemma, which may have
an independent interest.

\begin{lem}\label{lem3.10} Let $G$ be a commutative group scheme over an
algebraically closed field $k$ such that all its quotients in $\G_m$
are smooth. Let $\ell $ be a non-trivial character of $G$. Then $
H^1(G, \ell)=0$.
\end{lem}

\begin{proof} Let $\chi:G\to \G_m={\rm Aut}(\ell)$ be the non-trivial character,
and let $\sigma: G\to \ell$ be a cocycle representing a class in
$H^1( G  ,\ell)$. By definition of a cocycle,  one has
$\sigma(gh)=\chi(g) \sigma(h)+\sigma(g)$. The commutativity of $G$
implies
\ga{3.?}{\sigma(hg)=\chi(h)\sigma(g)+\sigma(h)=\sigma(gh)=\chi(g)
\sigma(h)+\sigma(g).} As $\chi$ is non-trivial, and ${\rm Im}
(\chi)\subset \G_m$ is assumed to be smooth, there is a $h\in G(k)$
such that $$0\neq (\chi(h)-1)\in {\rm End}(\ell)=k.$$ Thus
$$(\chi(h)-1)\in k^\times=\G_m(k)={\rm Aut}(\ell).$$ Set
$v=(\chi(h)-1)^{-1}\sigma(h)\in \ell$. Then, by \eqref{3.2}, one has
$\sigma(g)=\chi(g)v-v$, which means that $\sigma$ is a coboundary.
Thus
 $ H^1(G,\ell)=0$.
\end{proof}

\begin{rmk} \label{rmk3.11}
 The same proof shows that if $X$ is a smooth complex variety, with abelian
topological fundamental group, and if $\ell$ is a non-trivial rank
$1$ local system, then $H^1(X, \ell)=0$. Indeed, $G$ is now
$\pi_1^{{\rm top}}(X,a)$, the topological fundamental group based at
a complex point $a\in X(\C)$, and non-triviality implies the
existence of $h\in G$ with $(\chi(h)-1)\in \C^\times={\rm
Aut}(\ell)$. One then concludes identically.

This fact ought to be well known, but we could not find a reference
in the literature.
\end{rmk}

\begin{rmk} \label{rmk3.12}
As already mentioned in the introduction, the Malcev-Grothendieck
theorem  (\cite{Mal}, \cite{Gr}) asserts that the \'etale
fundamental group $\pi_1$ of a smooth complex projective variety is
trivial if and only its stratified  fundamental group $\pi^{{\rm
str}}$ is trivial, where $\pi^{\rm str}$ is the pro-affine complex
algebraic group of the Zariski closures of the monodromies of
complex linear representations of the topological fundamental group
$\pi_1^{\rm top}$. By going to the associated Galois cover, it
implies that $\pi^{{\rm str}}$ is finite if and only if $\pi_1$ is
finite (in which case $\pi^{{\rm str}}=\pi_1$). One also has that
$\pi^{{\rm str}}$ is abelian if and only if $\pi_1$ is abelian.
Indeed,  by definition $\pi^{{\rm str}}$ is abelian if and only
irreducible complex linear representations have dimension 1. But if
$\rho: \pi_1^{{\rm top}}\to GL(r, \C)$ is a representation, since
$\pi_1^{{\rm top}}$ is spanned by finitely many elements, $\rho$ has
values in $GL(r, A)$ for $A$ a ring of finite type over $\Z$. Then
if $\rho$ is irreducible, there is a closed point $s\in \Spec (A)$
such that $\rho\otimes \kappa(s): \pi_1^{\rm top}\to GL(r,
\kappa(s))$ is irreducible as well, where $\kappa(s)$ is the residue
field of $s$. As $\rho\otimes \kappa(s)$ is finite, it factors
through $\pi_1$, thus $r=1$.  Since a  linear representation with
finite monodromy is semi-simple, one can summarize as follows:
\begin{claim} \label{claim3.13}
 $\pi_1$ is
\begin{itemize}
\item[1)]  finite,
\item[2)] resp. abelian,
\item[3)] resp. abelian and
finite  \end{itemize}
if and only if $\pi_1^{{\rm str}}$
\begin{itemize}
\item[1)] finite
\item[2)] resp. abelian,
\item[3)] resp. abelian and
finite \end{itemize}
if and only if
\begin{itemize}
\item[1)]   local
systems have finite monodromy,
\item[2)]resp. irreducible local systems have rank $1$,
\item[3)] resp. local systems are direct sums of rank $1$
local systems. \end{itemize}
\end{claim}
While the main result of \cite{EM} is an analog in characteristic
$p>0$ of Claim~\ref{claim3.13} 1), the main theorem of this note is
an analog  of Claim~\ref{claim3.13} 2) and 3).

\end{rmk}

\bibliographystyle{plain}

\renewcommand\refname{References}

\end{document}